\documentclass[11pt]{amsart}
\usepackage{pxfonts}
\usepackage{amsmath}
\usepackage{amssymb}
\usepackage{amsthm}
\usepackage{amscd, amsfonts}
\usepackage[dvips]{epsfig}
\usepackage{verbatim}

\usepackage{epsf}
\usepackage[usenames]{color}
\usepackage[bookmarksnumbered,pdfpagelabels=true,plainpages=false,colorlinks=true,
            linkcolor=black,citecolor=black,urlcolor=blue]{hyperref}

\newcommand{\al}{\alpha}
\newcommand{\be}{\beta}

\newcommand{\ep}{\epsilon}

\newcommand{\ddc}{d d^{c}}

\theoremstyle{plain}
\newtheorem{theorem}{Theorem}[section]
\newtheorem{lemma}[theorem]{Lemma}

\theoremstyle{definition}
\newtheorem{rema}[theorem]{Remark}

\numberwithin{equation}{section}

\title[A PDE on a torus]{A fully nonlinear ``generalised Monge-Amp\`ere" PDE on a torus}

\author[Pingali]{Vamsi P. Pingali}
\address{Department of Mathematics\\
412 Krieger Hall, Johns Hopkins University,\\ Baltimore, MD 21218, USA}
\email{vpingali@math.jhu.edu}

\begin{document}
\maketitle

\begin{abstract}
We prove an existence result for a ``generalised" Monge-Amp\`ere equation introduced in \cite{GenMA} under some assumptions on a flat complex 3-torus. As an application we prove the existence of Chern connections on certain kinds of holomorphic vector bundles on complex 3-tori whose top Chern character forms are given representatives.
\end{abstract}

\section{Introduction}
\indent The complex Monge-Amp\`ere equation on a K\"ahler manifold was introduced by Calabi \cite{Calabi}, and was solved by Aubin \cite{Aub} and Yau \cite{Yau}. Since then other such fully nonlinear equations were studied, namely, the Hessian and inverse Hessian equations \cite{Hess1, Hess2, InvHess}. The inverse Hessian equations were introduced by X.X. Chen \cite{XX} in an attempt to find a lower bound on the Mabuchi energy. Actually, in \cite{XX} Chen conjectured that a fairly general fully nonlinear ``Monge-Amp\`ere" type PDE has a solution. Roughly speaking, instead of requiring the determinant of the complex Hessian of a function to be prescribed, it requires a combination of the symmetric polynomials of the Hessian to be given. A real version of such an equation was studied by Krylov \cite{Kryl} and a general existence result was proven by reducing it to a Bellman equation. In view of these developments a ``generalised Monge-Amp\`ere" equation was introduced in \cite{GenMA} and a few local ``toy models" were studied. As expected, the equation is quite challenging. The main problem is to find techniques to prove \emph{a priori} estimates in order to use the method of continuity to solve the equation. In this paper we study this equation on a flat complex torus wherein curvature issues do not play a role. The aim of this basic  example is to give insight into studying this equation in a more general setting. We prove an existence result (theorem \ref{Main}) in this paper.\\
\indent A small geometric application of this result is also provided - Given a $(k,k)$ form $\eta$ representing the $k$th Chern character class $[tr((\Theta)^k)]$ of a vector bundle on a compact complex manifold, it is very natural to ask whether there is a metric whose induced Chern connection realises $tr((\Theta)^k)=\eta$. As phrased this question seems almost intractable. It is not even obvious as to whether there is \emph{any} connection satisfying this requirement, leave aside a Chern connection. Work along these lines was done by Datta in \cite{Datta} using the h-principle. Therefore, it is more reasonable to ask whether equality can be realised for the top Chern character form. To restrict ourselves further we ask whether any given metric $h_0$ may be conformally deformed to $h_0 e^{-\phi}$ so as to satisfy a fully nonlinear PDE of the type treated in \cite{GenMA}. Admittedly, the result we have in this direction (theorem \ref{ChernWeil}) imposes quite a few restrictive assumptions on the type of vector bundles involved. However, the goal is to simply introduce the problem and solve it in a basic case to highlight the difficulties involved.

\section{Summary of results}
We prove an existence and uniqueness theorem for a ``generalised" Monge-Amp\`ere type equation \cite{GenMA} on a flat, complex 3-Torus. In whatever follows $\ddc = \sqrt{-1} \partial \bar{\partial}$ and $\omega _{f} = \omega + \ddc f$.
\begin{theorem}
Let $(X,\omega = \sqrt{-1}\omega _{i\bar{j}} dz^i \wedge d\bar{z}^j)$ be a flat, K\"ahler complex 3-torus (i.e. the $\omega _{i\bar{j}}$ are constants) $\frac{\mathbb{C}^3}{\Lambda}$ and $\alpha \geq \tilde{\ep} \omega \wedge \omega $ ($\tilde{\ep}>0$) be a smooth harmonic (i.e. constant coefficient) $(2,2)$ form on $X$ satisfying $\omega ^3 - \al \wedge \omega >0$. The following equation has a unique smooth solution $\phi$ satisfying $ 3(\omega + \ddc \phi)^2 - \al > 0$ and $\displaystyle \int _X \phi = 0$:
\begin{gather}
T(\phi) = \omega _{\phi} ^3 - \al \wedge \omega _{\phi}= \eta = e^F (\omega ^3 - \omega \wedge \alpha) > 0
\label{maineq}
\end{gather}
where $\displaystyle \int _X \eta = \int _X (\omega ^3 - \al \wedge \omega)$ and by $\al \geq \tilde{\ep} \omega \wedge \omega$ we mean that $(\al-\tilde{\ep} \omega \wedge \omega) = \displaystyle \sum _i \phi_i \wedge \bar{\phi}_i \wedge \Phi _i \wedge \bar{\Phi}_i$ for $(1,0)$-forms $\phi _i$ and $\Phi _i$.
\label{Main}
\end{theorem}

\begin{rema}
Let $\chi$ be a harmonic (with respect to $\omega$) K\"ahler form. Define $\tilde{\omega}$ as $\tilde{\omega} = \omega + \frac{\chi}{3}$ and assume that $3\tilde{\omega}^2 > 2 \chi \wedge \tilde{\omega}$.  As an interesting consequence one can see that the equation $\tilde{\omega} _{\phi} ^3 = \chi \wedge \tilde{\omega}_{\phi} ^2$ has a unique solution satisfying $\tilde{\omega}_{\phi}>0$ and $3\tilde{\omega} _{\phi}^2 > 2 \chi \wedge \tilde{\omega}_{\phi}$. This recovers existence for an inverse Hessian equation in this very special case. This also shows that solving the equation in general would give an alternate proof of existence for inverse Hessian equations , i.e., the results in \cite{InvHess}.
\end{rema}
A consequence of theorem \ref{Main} and the Calabi conjecture is the following theorem that deals with the existence of a Chern connection with a prescribed top Chern form.
\begin{theorem}
Let $X$ be a compact complex manifold of dimension $n$ and $(V,h_0)$ be a rank $k$ hermitian holomorphic vector bundle over $X$. In the following two cases, given an $(n,n)$ form $\eta$ representing the top Chern character class of $V$, there exists a smooth metric $h=h_0 e^{-2\pi\phi}$  such that its top Chern-Weil form of the Chern character class is $\eta$:
\begin{enumerate}
\item $X$ is a surface, i.e. $n=2$, $\mathrm{tr}(\Theta _0) >0$, and $(\mathrm{tr}(\Theta _0))^2  + k(\eta - \mathrm{tr}(\Theta _0 ^2)) >0$.
\item $X$ is a complex $3$-torus, $k \omega = \mathrm{tr}(\Theta _0)$ is a harmonic positive form, $\alpha = \frac{3(\mathrm{tr} (\Theta _0))^2}{k^2}-\frac{3\mathrm{tr}(\Theta _0 ^2)}{k} >0$ is harmonic, $-2(\mathrm{tr}(\Theta_0))^3+3k \mathrm{tr}(\Theta_0)\wedge \mathrm{tr}(\Theta_0^2)>0$, and $k^2 (\eta - \mathrm{tr}(\Theta _0 ^3) - 2(\mathrm{tr}(\Theta _0))^3+3k\mathrm{tr}(\Theta _0) \wedge \mathrm{tr}(\Theta _0 ^2)>0$
\end{enumerate}
where $\Theta _0 = \frac{\sqrt{-1}}{2\pi} F_0$ is the (normalised) curvature of the Chern connection associated to $h_0$.
\label{ChernWeil}
\end{theorem}
The hypotheses of theorem \ref{ChernWeil} require some discussion. As a warm-up example, let us consider the question for a line bundle, i.e., given a metric $h_0$ on a hermitian holomorphic line bundle $L$ on a complex $n$-fold with the curvature form denoted as $\Theta _{h_0}$, can we find a new metric $h=e^{-\phi} h_0$ such that the top Chern character form $(\frac{\sqrt{-1}}{2\pi}(\Theta_{h_0}+\ddc \phi)) ^n = \eta$ where $[\eta] = \mathrm{ch}_n (L)$ ? This is just the ``usual" Monge-Amp\`ere equation. To prove existence, the commonly made assumption is $\Theta _{h_0} >0$. So it is not at all surprising (and almost inevitable) that a ``generalised" version of such an equation would warrant more positivity assumptions, some of which might seem a little less geometric than desired. \\
\indent Nevertheless, here are a few examples (certainly not exhaustive) that satisfy the hypotheses :
\begin{enumerate}
\item $X$ is any compact complex surface, $(V,h_0)$ is any rank $k$ hermitian holomorphic vector bundle over $X$ such that $\mathrm{tr}(\Theta _0)>0$ and $\eta = (\mathrm{tr}(\Theta _0)) ^2 + \epsilon g \mathrm{tr}(\Theta _0)$ where $\int_X g \mathrm{tr}(\Theta _0) = 0$ and $\epsilon <<1$.
\item $X$ is the complex $3$-torus with the standard lattice $\mathbb{Z}\oplus \mathbb{Z}\oplus \mathbb{Z}$. Choose three line bundles $(L_1,h_1), (L_2,h_2), (L_3,h_3)$ so that their Chern forms are $\omega _1 = \sqrt{-1} \sum dz^i \wedge d\bar{z}^i$, $\omega _2 = \sqrt{-1}(3dz^2 \wedge d\bar{z}^2 + dz^3 \wedge d\bar{z}^3)$, $\omega _3 = 2dz^3 \wedge d\bar{z}^3$. Take $(V,h_0)$ to be their direct sum and $\eta = \mathrm{tr}(\Theta _0 ^3) + \epsilon g$ where $\epsilon <<1$ and $\int g=0$.
\end{enumerate}

\section{Proofs of the Theorems}
\indent We first prove a useful lemma :
\begin{lemma}
Let $X$ be a K\"ahler $3$-manifold. If $\gamma$ is a non-negative real $(1,1)$ form and $\beta$ be a strongly strictly positive real $(2,2)$ form (hence $*\beta >0$ for the Hodge star of any K\"ahler metric) such that $\gamma ^3 - \beta \wedge \gamma >0$ then $3\gamma ^2 - \beta >0$ and $\gamma >0$.
\label{useful}
\end{lemma}
\emph{Proof} : Since $\gamma ^3 >0$ it is clear that $\gamma >0$.  Let $*$ denote the Hodge star with respect to $\gamma$. Notice that $\beta \wedge \gamma = *\beta \wedge \frac{\gamma ^2}{2}$. Since we are dealing with top forms, we may divide by $(*\beta) ^3$ to get $\frac{\gamma ^3}{(*\beta)^3} - \frac{*\beta \wedge  \gamma ^2}{2(*\beta)^3} >0$. At a point $p$, choose coordinates so that the strictly positive form $*\beta$ is $\sqrt{-1}\sum dz^i \wedge d\bar{z}^i$ and $\gamma$ is diagonal with eigenvalues $\lambda _i$. Then at $p$, $6\lambda _1 \lambda _2 \lambda _3 - (\sum _{i<j} \lambda _i \lambda _j) >0$ thus implying that $6\lambda _i >1$. This means $6\gamma - *\beta >0$. Applying $*$ we see that $3\gamma ^2 - \beta >0$.   \\
\newline
\indent We need another lemma
\begin{lemma} 
Let $X$ be a K\"ahler $3$-manifold. If $\gamma$ is a positive real $(1,1)$ form, $\eta>0$ is a $(3,3)$ form, and $\beta$ be a strongly strictly positive real $(2,2)$ form, then the functions $\mathcal{F}:\gamma \rightarrow \frac{\beta \wedge \gamma}{\gamma ^3}$ and $\mathcal{G}:\gamma \rightarrow \frac{\eta}{\gamma^3}$ are convex.
\label{EK}
\end{lemma}
\emph{Proof} : Fix a K\"ahler form $\omega$ for $X$ and let $*$ be its Hodge star. Choose coordinates so that $\omega = \sqrt{-1}\sum dz^i \wedge d\bar{z}^i$ at a point $p$. By a linear change of coordinates $*\beta$ may be diagonalised at $p$. Hence $\beta = -b_3 dz^1 \wedge d\bar{z}^1 \wedge dz^2 \wedge d\bar{z}^2 -  b_2 dz^3 \wedge d\bar{z}^3 \wedge dz^1 \wedge d\bar{z}^1 -  b_1 dz^2 \wedge d\bar{z}^2 \wedge dz^3 \wedge d\bar{z}^3$ at $p$. By scaling $z_i$ appropriately we may assume that $b_i=1$. At $p$ the function $\mathcal{F}$ is $A \rightarrow \frac{\mathrm{tr}(A)}{6\det(A)}$ where $A$ is a positive hermitian matrix. The fact that this and $G(A) = \frac{1}{\det(A)}$ are convex is proven in \cite{Kryl}. ( Notice that $\mathcal{G}(A) = K G(A)$ for some positive constant $K$.) \\

It is easy to see that the set $\mathcal{S}$ of $\gamma>0$ in lemma \ref{useful} satisfying $\gamma ^3 - \beta \wedge \gamma >0$ is a convex open set. In fact, a stronger statement holds. 
\begin{lemma}
Let $\gamma _1$, $\gamma _2$ lie in $\mathcal{S}$ and $\gamma _t = t\gamma _1 + (1-t) \gamma _2$. Then $3\gamma_t ^2 - \beta > Ct \gamma _1 ^2$ where $C$ depends only on $\gamma _1$ and $\beta$.
\label{rem1}
\end{lemma}
\begin{proof}
Notice that 
\begin{gather}
\gamma _t ^3 - \beta \wedge \gamma_t = \gamma _t ^3 (1- \frac{\beta \wedge \gamma _t}{\gamma_ t ^3}) \nonumber \\
\geq \gamma _t ^3 \left(1- t\frac{\beta \wedge \gamma _1}{\gamma_ 1 ^3} - (1-t)\frac{\beta \wedge \gamma _2}{\gamma_ 2 ^3}\right) \label{in1}
\end{gather}
where the last inequality follows from lemma \ref{EK}. Since $\gamma _1$ and $\gamma _2$ lie in $\mathcal{S}$, 
\begin{gather}
\gamma _t ^3 \left(1- t\frac{\beta \wedge \gamma _1}{\gamma_ 1 ^3} - (1-t)\frac{\beta \wedge \gamma _2}{\gamma_ 2 ^3}\right) \geq t\gamma _t ^3 \left(1- \frac{\beta \wedge \gamma _1}{\gamma_ 1 ^3} \right) \nonumber \\
> \tilde{C}t\gamma _t ^3 \label{in2}
\end{gather}
where $\tilde{C}$ is a small positive constant depending only on $\gamma_1$ and $\beta$. Putting \ref{in1} and \ref{in2} together we have
\begin{gather}
\gamma _t ^3 - \frac{\beta}{1-\tilde{C}t} \wedge \gamma_t > 0 \nonumber 
\end{gather}
This implies (by using lemma \ref{useful}) that 
\begin{gather}
3\gamma _t ^2 - \frac{\beta}{1-\tilde{C}t} > 0 \nonumber \\
\Rightarrow 3\gamma _t ^2 - \beta > \frac{\tilde{C}t}{1-\tilde{C}t}\beta > Ct\gamma_1 \nonumber  
\end{gather}
\end{proof}
\subsection{Proof of theorem \ref{Main}}:  
We use the method of continuity. Consider the family of equations for $t$ in  $[0,1]$
\begin{gather}
(\omega + \ddc \phi_t)^3-\alpha \wedge (\omega +\ddc \phi_t) = \frac{e^{tF} \int _X (\omega ^3 - \alpha \wedge \omega)}{\int _X e^{tF} (\omega ^3 - \alpha \wedge \omega)}(\omega ^3 - \alpha \wedge \omega)
\label{conti}
\end{gather}
At $t=0$, $\phi=0$ is a solution. By lemma \ref{useful} ellipticity is preserved along the path. We verify that theorem $2.1$ of \cite{GenMA} applies here. Indeed, we notice that $T(\phi)-T(0)=\int _0 ^1 \frac{dT(t\phi)}{dt} dt = \ddc \phi \wedge \int _0 ^1 (3\omega _{t\phi} ^2 -\al) dt$ and that lemma \ref{rem1} (along with the substitution $\tilde{t}=1-t$ in the integral) implies that the conditions of theorem $2.1$ are satisfied. This proves that the set of $t$ for which solutions exist is open, solutions are unique and have an \emph{a priori} $C^0$ bound. To prove that it is closed we need $C^{2,\be}$ \emph{a priori} estimates (by Schauder theory this is enough to bootstrap the regularity). We proceed to find such estimates for $(\omega + \ddc \phi)^3 - \al \wedge (\omega +\ddc \phi) = f\omega ^3$. Locally $\omega =\sqrt{-1}\sum dz^i \wedge d\bar{z}^i$, $u = \sum \vert z \vert^2 + \phi$, and
\begin{gather}
\det(\ddc u) - tr(A\ddc u) = f 
\label{loc}
\end{gather}
for some hermitian positive matrix $A$. If $\alpha$ is diagonalised such that $\alpha =  dz ^1 \wedge d\bar{z}^1 \wedge dz^2 \wedge d\bar{z}^2 + \ldots$ then $A = \frac{1}{6} \mathrm{Id}$.\\
\newline
\emph{$C^1$ estimate}: \\
\indent For this we shall not make the assumption that $\al$ is harmonic. This assumption will be used only in the higher order estimates. Following \cite{Blockigradient} let O be a point where $\beta = \ln (\vert \nabla \phi \vert^2) - \gamma (\phi)$ achieves its maximum. (If we prove that $\beta$ is bounded, then so is the first derivative. So assume that $\vert \nabla \phi \vert > 1$ without loss of generality. $\beta$ is Blocki's function. $\gamma$ will be chosen later.) Differentiating once we see that $\det(\ddc u) tr((\ddc u)^{-1} (\ddc u_k) ) - tr(A_{,k} \ddc u) - tr(A\ddc u_k) = f_k$ (and similarly for $\bar{k}$). Let $L$ be the matrix $\det(\ddc u) (\ddc u)^{-1} - A >0$. Hence 
\begin{gather}
tr(L \ddc u_i) = f_i + tr(A_{,i} \ddc u)
\label{diffonce}
\end{gather}
 At O we may assume that $\phi_{i\bar{j}}$ is diagonal. Besides, $\beta _k =0$ there and $tr(L \beta _{k\bar{l}}) \leq 0$ at O. The first condition implies that
\begin{gather}
\frac{1}{\vert \nabla \phi \vert^2}(\sum \phi _{ik} \phi _{\bar{i}}+\phi_i \phi_{\bar{i}k}) -\gamma ^{'} \phi _k = 0 
\label{beta1}
\end{gather}
at O. Moreover,
\begin{gather}
\beta _{k\bar{l}} = -\frac{1}{\vert \nabla \phi \vert ^4} (\sum \phi _{ik} \phi _{\bar{i}}+\phi_i \phi_{\bar{i}k})(\sum \phi _{j\bar{l}}\phi _{\bar{j}} + \phi_j \phi _{\bar{j}\bar{l}}) + \nonumber \\ \frac{1}{\vert \nabla \phi \vert^2}(\sum \phi _{ik\bar{l}}\phi_{\bar{i}} + \phi _{ik}\phi _{\bar{i}\bar{l}} + \phi_{i\bar{l}}\phi_{\bar{i}k}+\phi_i \phi_{\bar{i}k \bar{l}})-\gamma ^{''}\phi_{\bar{l}}\phi_{k} -\gamma ^{'}\phi_{k\bar{l}}
\label{beta2}
\end{gather}
Noticing that $\ddc u_i = \ddc \phi _i$ and using equations \ref{beta1} and \ref{diffonce} we get (at O)
\begin{gather}
0\geq tr(L\beta_{k \bar{l}}) = -((\gamma ^{'})^2 +\gamma ^{''} )tr(L\phi_k \phi_{\bar{l}}) + \frac{1}{\vert \nabla \phi \vert^2}(\sum \phi_{\bar{i}}(f_i+tr(A_{,i}\ddc u)) + \phi _i (f_{\bar{i}}+ tr(A_{,\bar{i}}\ddc u))+ \nonumber \\ tr(L\phi_{ik}\phi_{\bar{i}\bar{l}}) + tr(L\phi_{i\bar{l}}\phi_{\bar{i}k})) -\gamma ^{'} tr(L\ddc u) + \gamma ^{'} tr(L) \nonumber \\
\geq -C_{\gamma, \vert \phi \vert_{C^0}}+tr(-L((\gamma ^{'})^2 +\gamma ^{''})\phi _k \phi _{\bar{l}})-2\gamma ^{'}tr(A\ddc u) -\frac{C}{\vert \nabla \phi \vert}tr(A\ddc u) + \gamma ^{'} \sum \frac{f+tr(A\ddc u)}{u_{i\bar{i}}}
\end{gather}
If we choose $\gamma$ so that $\gamma ^{'}>E>0$,  and $-((\gamma ^{'})^2+\gamma^{''})>Q>0$ (where $E$ and $Q$ are arbitrary positive constants), then this forces $(\ddc u)^{-1} (O)$ to be bounded. For instance $\gamma$ can be chosen \cite{Blockigradient} to be $\gamma (x) = \frac{1}{2}\ln(2x+1)$. Assume that $\vert \nabla \phi \vert \rightarrow \infty$. If $\sum \frac{1}{u_{i\bar{i}}} >2+\epsilon$ uniformly then surely $\Delta u (O)$ is bounded. This observation actually implies that $\Delta u (O)$ is bounded. Indeed,
\begin{lemma}
At any point $Q$ if $\Delta u \rightarrow \infty$, then $\sum \frac{1}{u_{i\bar{i}}} >2+\epsilon$ for some uniform $\epsilon$
\label{inter}
\end{lemma}
\begin{proof}
Choose normal coordinates for $\omega$ around $Q$ so that $\ddc u$ is diagonal at $Q$. Recall that $\omega ^3 - \al \wedge \omega > \tilde{\ep}\omega ^3 $ forces $A_{ii} < 1-\tilde{\ep}$. Let $u_{i\bar{i}} (Q)=\lambda_i$ with $\lambda _1 \geq \lambda _2 \geq \lambda_3 \geq C>0$. (If $\lambda _3$ gets arbitrarily close to $0$, then the lemma is obviously true.) If $\lambda _1 \rightarrow \infty$ it is clear from the equation $\lambda _1 \lambda _2 \lambda _3 = f + \sum A_{ii} \lambda _i$ that $\lambda _3$ should be bounded. Solving for $\lambda_1$, one can see that $\lambda _2 \rightarrow \frac{A_{11}}{\lambda _3}$. This means that $\sum \frac{1}{\lambda _i}$ goes to $\frac{1}{\lambda _3} + \frac{\lambda _3}{A_{11}} \geq 2 (1/A_{11})^{1/2} > 2+\epsilon$. 
\end{proof}

Lemma \ref{inter} implies that $L$ is bounded below and above at O. This means that $\nabla \phi$ is bounded at O. \\
\newline 
\emph{$C^{1,1}$ estimate} : \\ 
\indent Define $g = \frac{\alpha \wedge \omega _{\phi}}{\omega^3} - \phi$. Locally $g=\mathrm{tr}(A\ddc u) - \phi$. If $g$ is bounded, then thanks to the previous $C^0$ estimate on $\phi$, so is $\mathrm{tr}(A\ddc u)$. This will give us the desired bound on $\Delta \phi$ and hence on $\ddc \phi$, i.e. the $C^{1,1}$ estimate. \\
\indent Differentiating equation \ref{loc} we see that
\begin{gather}
\mathrm{tr}\left ( (\det(\ddc u) (\ddc u)^{-1} - A) \ddc u_k\right ) = f_k \nonumber \\
\Rightarrow \mathrm{tr}(L\ddc u_k) = f_k \nonumber 
\label{locdiffone}
\end{gather}
where the matrix $L=\det(\ddc u) (\ddc u)^{-1} - A > 0$ as before. Differentiating again and taking the trace \footnote{In whatever follows, upper indices do not denote the inverse matrix. They just denote the original matrix itself and are used to make the Einstein summation convention work nicely.} after multiplication with  $A$ we see that
\begin{gather}
A^{k\bar{l}}\mathrm{tr}\left ( (\det(\ddc u) (\ddc u)^{-1} - A) \ddc u_{k\bar{l}}\right ) = A^{k\bar{l}}f_{k\bar{l}} + \det(\ddc u)   A^{k\bar{l}}\mathrm{tr}\left ( (\ddc u)^{-1}\ddc u_{\bar{l}} (\ddc u)^{-1} \ddc u_k  \right )  \nonumber \\ - \det(\ddc u)A^{k\bar{l}} \mathrm{tr} \left ( (\ddc u)^{-1} \ddc u_{\bar{l}} \right ) \left ( (\ddc u)^{-1} \ddc u_{k} \right ) \nonumber \\
\Rightarrow A^{k\bar{l}}\mathrm{tr}\left ( L \ddc u_{k\bar{l}}\right ) = A^{k\bar{l}}f_{k\bar{l}} + \det(\ddc u)   A^{k\bar{l}}\mathrm{tr}\left ( (\ddc u)^{-1}\ddc u_{\bar{l}} (\ddc u)^{-1} \ddc u_k  \right )  \nonumber \\ - \det(\ddc u)A^{k\bar{l}} \mathrm{tr} \left ( (\ddc u)^{-1} \ddc u_{\bar{l}} \right ) \mathrm{tr}\left ( (\ddc u)^{-1} \ddc u_{k} \right ) 
\label{locdiff}
\end{gather}
 Upon differentiating $g$ we see that 
\begin{gather}
g_k = \mathrm{tr}(A \ddc u_k) - \phi _k\\
g_{k\bar{l}} =  \mathrm{tr}(A \ddc u_{k\bar{l}}) - \phi _{k\bar{l}}
\label{gdiff}
\end{gather}
Let us assume that $g$ attains its maximum at a point $P$. At $P$, $g_k=0$, $u=\phi$, $u_k = \phi _k$, $u_{k\bar{l}} = \phi_{k \bar{l}} + \delta_{k\bar{l}}$, and $\mathrm{tr}(L[g_{k\bar{l}}])=L^{k\bar{l}}g_{k\bar{l}}\leq 0$. Choose normal coordinates for $\omega$ around $P$ so that $\ddc u$ is diagonal at $P$. Putting these observations, and equations \ref{locdiffone}, \ref{locdiff} and \ref{gdiff} together we see that at $P$ (all the arbitrary constants that occur below are positive by convention)
\begin{gather}
0\geq -L^{k\bar{l}} \phi _{k \bar{l}} + A^{k\bar{l}}f_{k\bar{l}} + \det(\ddc u)   A^{k\bar{l}}\mathrm{tr}\left ( (\ddc u)^{-1}\ddc u_{\bar{l}} (\ddc u)^{-1} \ddc u_k  \right )  \nonumber \\ - \det(\ddc u)A^{k\bar{l}} \mathrm{tr} \left ( (\ddc u)^{-1} \ddc u_{\bar{l}} \right ) \mathrm{tr} \left ( (\ddc u)^{-1} \ddc u_{k} \right ) \\
\geq -L^{k\bar{l}} u _{k \bar{l}}+\mathrm{tr}(L) + A^{k\bar{l}} f_{k\bar{l}} - \det(\ddc u)A^{k\bar{l}} \mathrm{tr} \left ( (\ddc u)^{-1} \ddc u_{\bar{l}} \right ) \mathrm{tr} \left ( (\ddc u)^{-1} \ddc u_{k} \right ) \\
\geq -3\det(\ddc u)+ \mathrm{tr}(L)+ A^{k\bar{l}}u_{k\bar{l}} - C - A^{k\bar{l}}\frac{(f_k + \mathrm{tr}(A \ddc u_k))(f_l + \mathrm{tr}(A \ddc u_l))}{\det(\ddc u)} \nonumber \\
\geq -2 A^{k\bar{l}}u_{k\bar{l}} + \mathrm{tr}(L) - C - A^{k\bar{l}}\frac{(f_k + u_k)(f_l + u_l)}{f+\mathrm{tr}(A \ddc u)} \nonumber \\
\geq -2 A^{k\bar{l}}u_{k\bar{l}} + \det(\ddc u) \mathrm{tr}((\ddc u )^{-1}) - C_1 - \frac{C_2}{f+\mathrm{tr}(A \ddc u)} \nonumber \\
= -2 A^{k\bar{l}}u_{k\bar{l}} + (f+\mathrm{tr}(A\ddc u)) \mathrm{tr}((\ddc u )^{-1}) - C_1 - \frac{C_2}{f+\mathrm{tr}(A \ddc u)}
\label{finaldiff}
\end{gather}
Let $u_{l\bar{l}}$ at $P$ be $\lambda _l$. Thus at $P$
\begin{gather}
0 \geq -2 \displaystyle \sum _{l=1} ^{3} A_{l \bar{l}}\lambda _l +  \sum _{l=1} ^{3} A_{l \bar{l}}\lambda _l \sum_{k=1} ^{3} \frac{1}{\lambda _k} - C_1 - \frac{C_2}{f+\mathrm{tr}(A \ddc u)}  \nonumber \\
= \displaystyle \left ( \sum_{k=1} ^{3} \frac{1}{\lambda _k} - 2\right ) \sum _{l=1} ^{3} A_{l \bar{l}}\lambda _l - C_1 - \frac{C_2}{f+\mathrm{tr}(A \ddc u)}  
\label{finaldifftwo}
\end{gather}
Using lemma \ref{inter} we see that if $\Delta u \rightarrow \infty$ at $P$, then 
\begin{gather}
0 \geq \displaystyle \epsilon \sum _{l=1} ^{3} A_{l \bar{l}}\lambda _l - C_1 - \frac{C_2}{f+\mathrm{tr}(A \ddc u)}
\label{finaldiffthree}
\end{gather}
It is clear from equation \ref{finaldiffthree} that $\mathrm{tr}(A \ddc u)$ is bounded at $P$ and hence so is $g$. As mentioned earlier this implies the desired $C^{1,1}$ estimate. \\
\newline 
\emph{$C^{2,\be}$ estimate}: \\
\indent Rewriting the equation (just as in \cite{GenMA}) $-1=-\frac{\eta}{(\omega + \ddc \phi)^3} - \frac{\al \wedge (\omega + \ddc \phi)}{(\omega + \ddc \phi)^3}$ and using lemma \ref{EK} we see that the (complex version \cite{Blocki}\cite{Siu} of) Evans-Krylov theory applies to it. This proves the desired estimate. \\
\newline
\subsection{Proof of theorem \ref{ChernWeil}}  
The curvature $\Theta (h) = \Theta _0 + \ddc \phi$. Hence $\mathrm{tr}((\Theta _0 + \ddc \phi)^n)=\eta$. This equation reduces in the two cases of the theorem to
\begin{gather}
(\ddc \phi + \frac{\mathrm{tr}(\Theta _0)}{k})^2 = \frac{\eta-\mathrm{tr}((\Theta _0)^2)}{k}+\frac{(\mathrm{tr}(\Theta _0))^2}{k^2} \\
(\ddc \phi + \frac{\mathrm{tr}(\Theta _0)}{k})^3 - (\ddc \phi + \frac{\mathrm{tr}(\Theta _0)}{k}) \wedge (\frac{-3\mathrm{tr}(\Theta _0 ^2)}{k}+3\frac{(\mathrm{tr}(\Theta _0))^2}{k^2})  = \nonumber \\
  \frac{\eta - \mathrm{tr}(\Theta _0 ^3)}{k}-\frac{2(\mathrm{tr}(\Theta _0))^3-3k\mathrm{tr}(\Theta _0) \wedge \mathrm{tr}(\Theta _0 ^2)}{k^3}
\end{gather}
respectively. The first equation may be solved under the given hypotheses using Aubin-Yau's solution \cite{Yau}\cite{Aub} of the Calabi conjecture \cite{Calabi}. The second one is solved using theorem \ref{Main}.

\end{document}